\documentclass[12pt, reqno]{article}
\usepackage{amssymb,amsmath,amsthm,graphicx,multirow,setspace,url, enumerate, float, amscd, fullpage, hyperref}
\numberwithin{equation}{section}

\begin{document}

\theoremstyle{plain}

\newtheorem{conjecture}{Conjecture}
\newtheorem{theorem}{Theorem}[section]
\newtheorem{lemma}[theorem]{Lemma}
\newtheorem{corollary}[theorem]{Corollary}

\theoremstyle{definition}
\newtheorem{definition}[theorem]{Definition}
\newtheorem{remark}{Remark}[section]
\newtheorem{example}{Example}[section]

\newtheorem*{theorem*}{Theorem}
\newtheorem*{corollary*}{Corollary}

\begin{center}
\vskip 1cm{\LARGE\bf 
Counting non-standard\\
\vskip .12in
binary representations
}
\vskip 1cm
\large
Katie Anders\footnote{The author acknowledges support from National Science
Foundation grant DMS 08-38434``EMSW21-MCTP: Research Experience for
Graduate Students''.} \\
Department of Mathematics\\
University of Texas at Tyler\\
3900 University Blvd.\\
Tyler, TX 75799\\
USA\\
\href{mailto:kanders@uttyler.edu}{\tt kanders@uttyler.edu}
\end{center}

\vskip .2 in

\begin{abstract}
Let $\mathcal{A}$ be a finite subset of $\mathbb{N}$ including $0$ and $f_\mathcal{A}(n)$ be the number of ways to write $n=\sum_{i=0}^{\infty}\epsilon_i2^i$, where $\epsilon_i\in\mathcal{A}$.  We consider asymptotics of the summatory function $s_\mathcal{A}(r,m)$ of $f_\mathcal{A}(n)$ from $m2^r$ to $m2^{r+1}-1$ and show that $s_{\mathcal{A}}(r,m)\approx c(\mathcal{A},m)\left|\mathcal{A}\right|^r$ for some $c(\mathcal{A},m)\in\mathbb{Q}$.

\end{abstract}

\section{Introduction}\label{Introduction}
Let $f_{\mathcal{A}}(n)$ denote the number of ways to write $n=\sum_{i=0}^{\infty}\epsilon_i2^i$, where $\epsilon_i$ belongs to the set
\[
\mathcal{A}:=\{0=a_0, a_1,\ldots,a_z\}, 
\]
with $a_i\in\mathbb{N}$ and $a_i< a_{i+1}$ for all $0\leq i\leq z-1$.  For more on this topic, see the author's previous work \cite{my first paper}.  We parameterize $\mathcal{A}$ in terms of its $s$ even elements and $(z+1)-s:=t$ odd elements as follows:
\[
\mathcal{A}=\{0=2b_1,2b_2,\ldots,2b_s,2c_1+1,\ldots,2c_t+1\}.
\]

If $n$ is even, then $\epsilon_0=0,2b_2,2b_3,\ldots$, or $2b_s$ and
\[
 f_\mathcal{A}(n)=f_\mathcal{A}(n/2)+f_\mathcal{A}((n-2b_2)/2)+f_\mathcal{A}((n-2b_3)/2)+\cdots+f_\mathcal{A}((n-2b_s)/2).
\]
Writing $n=2\ell$, we have
\[
f_\mathcal{A}(2\ell)=f_\mathcal{A}(\ell)+f_\mathcal{A}(\ell-b_2)+f_\mathcal{A}(\ell-b_3)+\cdots+f_\mathcal{A}(\ell-b_s),
\]
so for any even $n$, $f_{\mathcal{A}}(n)$ satisfies a recurrence relation of order $b_s$.

Similarly, if $n=2\ell+1$ is odd, then $\epsilon_0=2c_1+1,2c_2+1,\ldots,$ or $2c_t+1$, and
\[
f_\mathcal{A}(2\ell+1)=f_\mathcal{A}(\ell-c_1)+f_\mathcal{A}(\ell-c_2)+\cdots+f_\mathcal{A}(\ell-c_t),
\]
so for any odd $n$, $f_{\mathcal{A}}(n)$ satisfies a recurrence relation of order $c_t$.  Dennison, Lansing, Reznick, and the author \cite{4 paper} gave this argument for $f_{\mathcal{A},b}(n)$, the $b$-ary representation of $n$ with coefficients from $\mathcal{A}$, using residue classes $\operatorname{mod} b$.

\begin{example}\label{0134 even and odd recurrences}
Let $\mathcal{A}=\{0,1,3,4\}$.  We can write $\mathcal{A}=\{2(0),2(0)+1,2(1)+1, 2(2)\}$.  Then
\begin{equation}\label{0134 recurrences}
f_{\mathcal{A}}(2\ell)=f_{\mathcal{A}}(\ell)+f_{\mathcal{A}}(\ell-2)\quad\text{ and }\quad f_{\mathcal{A}}(2\ell+1)=f_{\mathcal{A}}(\ell)+f_{\mathcal{A}}(\ell-1).
\end{equation}
\end{example}

In general, let 
\[
\omega_k(m)=
\left(
\begin{array}{c}
f_{\mathcal{A}}(2^km)\\
f_{\mathcal{A}}(2^km-1)\\
\vdots\\
f_{\mathcal{A}}(2^km-a_z)
\end{array}
\right).
\]
We shall consider the fixed $(a_z+1)\times(a_z+1)$ matrix $M_{\mathcal{A}}$ such that for any $k\geq 0$,
\[
 \omega_{k+1}=M_{\mathcal{A}}\omega_k.
\]

\begin{example}
Returning to the set $\mathcal{A}=\{0,1,3,4\}$ of Example \ref{0134 even and odd recurrences} and using the equations in \eqref{0134 recurrences}, we have
\begin{align}\label{defining Ma}
 \omega_{k+1}(m)&=
\left(
\begin{array}{c}
f_{\mathcal{A}}(2^{k+1}m)\\
f_{\mathcal{A}}(2^{k+1}m-1)\\
f_{\mathcal{A}}(2^{k+1}m-2)\\
f_{\mathcal{A}}(2^{k+1}m-3)\\
f_{\mathcal{A}}(2^{k+1}m-4)\\
\end{array}
\right)=
\left(
\begin{array}{c}
f_{\mathcal{A}}(2^km)+f_{\mathcal{A}}(2^km-2)\\
f_{\mathcal{A}}(2^km-1)+f_{\mathcal{A}}(2^km-2)\\
f_{\mathcal{A}}(2^km-1)+f_{\mathcal{A}}(2^km-3)\\
f_{\mathcal{A}}(2^km-2)+f_{\mathcal{A}}(2^km-3)\\
f_{\mathcal{A}}(2^km-2)+f_{\mathcal{A}}(2^km-4)\\
\end{array}
\right)\\
&=\left(
\begin{array}{c c c c c}
1 &0 &1 &0 &0\\
0 &1 &1 &0 &0\\
0 &1 &0 &1 &0\\
0 &0 &1 &1 &0\\
0 &0 &1 &0 &1
\end{array}
\right)
\left(
\begin{array}{c}
f_{\mathcal{A}}(2^km)\\
f_{\mathcal{A}}(2^km-1)\\
f_{\mathcal{A}}(2^km-2)\\
f_{\mathcal{A}}(2^km-3)\\
f_{\mathcal{A}}(2^km-4)\\
\end{array}
\right)\nonumber
\end{align}
If $M_{\mathcal{A}}$ is the matrix in \eqref{defining Ma}, then $\omega_{k+1}(m)=M_{\mathcal{A}}\omega_k(m)$.
\end{example}

We now review some basic concepts of sequences from Section 1.8 of Lidl and Niederreiter \cite{Lidl&N} and include a matrix view of recurrence relations, following Reznick \cite{Stern notes}.

Consider a sequence $\left(b(n)\right)$ such that
\begin{equation}\label{kth order recurrence}
b(n)+c_{k-1}b(n-1)+c_{k-2}b(n-2)+\cdots+c_0b(n-k)=0
\end{equation}
for all $n\geq k$ and $c_i\in\mathbb{N}$.
By shifting the sequence, we see that
\begin{equation}\label{kth order recurrence shifted}
b(n+k)+c_{k-1}b(n+k-1)+c_{k-2}b(n+k-2)+\cdots+c_0b(n+k-k)=0
\end{equation}
for $n\geq 0$.
Then \eqref{kth order recurrence} is a \textit{homogeneous k-th order linear recurrence relation}, and $\left(b(n)\right)$ is a \textit{homogeneous k-th order linear recurrence sequence}.  The coefficients $c_0,c_1,\ldots,c_{k-1}$ are the \textit{initial values of the sequence}.  For any sequence $\left(b(n)\right)$ satisfying \eqref{kth order recurrence} we define the \textit{characteristic polynomial}
\begin{equation}\label{char poly}
 f(x)=x^k+c_{k-1}x^{k-1}+c_{k-2}x^{k-2}+\cdots+c_0.
\end{equation}

We can also consider a recurrence relation from the point of view of a matrix system, considering $k$ sequences indexed as $\left(b_i(n)\right)$ for $1\leq i\leq k$ which satisfy
\begin{equation*}
b_i(n+1)=\sum_{j=1}^k m_{ij}b_j(n)
\end{equation*}
for $n\geq 0$ and $1\leq i\leq k$.  Then
\begin{equation*}
\left(
\begin{array}{c}
b_1(n+1)\\
\vdots\\
b_k(n+1)
\end{array}
\right)
=
\left(
\begin{array}{c c c}
m_{11} &\cdots &m_{1k} \\
\vdots & &\vdots\\
m_{k1} &\cdots &m_{kk}
\end{array}
\right)
\left(
\begin{array}{c}
b_1(n)\\
\vdots\\
b_k(n)
\end{array}
\right)
\end{equation*}
for $n\geq 0$.  To simplify the notation, if $M=[m_{ij}]$ and
\begin{equation*}
\mathbf{B}(n)=
\left(
\begin{array}{c}
b_1(n)\\
\vdots\\
b_k(n)
\end{array}
\right),
\end{equation*}
then $\mathbf{B}(n+1)=M\mathbf{B}(n)$ for $n\geq 0$.
Thus $\mathbf{B}(n)=M^n\mathbf{B}(0)$ for $n\geq 0$, where
\begin{equation*}
\mathbf{B}(0)=
\left(
\begin{array}{c}
b_1(0)\\
\vdots\\
b_k(0)
\end{array}
\right)
\end{equation*}
is the vector of initial conditions.

In this matrix point of view, the \textit{characteristic polynomial} of $M$ is
\begin{equation*}
g(\lambda):=\operatorname{det}(M-\lambda I_k).
\end{equation*}
By the Cayley-Hamilton Theorem, $g(M)=\mathbf{0}$, the $k\times k $ zero matrix.

If $g(x)$ is the characteristic polynomial in \eqref{char poly}, then
\begin{equation*}
 \mathbf{0}=g(M)=M^k+c_{k-1}M^{k-1}+c_{k-2}M^{k-2}+\cdots+c_0I_k.
\end{equation*}
Hence for any $n\geq 0$,
\begin{equation*}
\mathbf{0}=M^{n+k}+c_{k-1}M^{n+k-1}+c_{k-2}M^{n+k-2}+\cdots+c_0M^n 
\end{equation*}
and thus
\begin{align*}
\mathbf{0}&=\left(M^{n+k}+c_{k-1}M^{n+k-1}+c_{k-2}M^{n+k-2}+\cdots+c_0M^n\right)\mathbf{B}(0)\\ 
 &=\mathbf{B}(n+k)+c_{k-1}\mathbf{B}(n+k-1)+c_{k-2}\mathbf{B}(n+k-2)+\cdots+c_0\mathbf{B}(n).
\end{align*}
Thus each sequence $\left(b_j(n)\right)$ satisfies the original linear recurrence \eqref{kth order recurrence shifted}.

As an additional connection between these two views of linear recurrence sequences, note that for a sequence satisfying \eqref{kth order recurrence},
\begin{equation*}
\left(
\begin{array}{c}
b(n+1)\\
b(n+2)\\
\vdots\\
b(n+k-1)\\
b(n+k)
\end{array}
\right)
=
\left(
\begin{array}{c c c c c}
0 &1 &\cdots &0 &0 \\
0 &0 &\cdots &0 &0\\
\vdots &\vdots & &\vdots &\vdots\\
0 &0 &\cdots &0 &1\\
-c_0 &-c_{1} &\cdots &-c_{k-2} &-c_{k-1}
\end{array}
\right)
\left(
\begin{array}{c}
b(n)\\
b(n+1)\\
\vdots\\
b(n+k-2)\\
b(n+k-1)
\end{array}
\right),
\end{equation*}
where this matrix, the \textit{companion matrix} to $g$, has characteristic polynomial $(-1)^kg$.

\section{Main Result}\label{Main Result}

We will use the ideas  of Section \ref{Introduction} to examine the asymptotic behavior of the summatory function $\displaystyle\sum_{n=m2^r}^{m2^{r+1}-1}f_{\mathcal{A}}(n)$, but we must first establish a lemma.

\begin{lemma}[{\cite[5.6.5 \& 5.6.9]{Matrix Theory}}]\label{eignevalues bounded by row sums}
Let $M=[m_{ij}]$ be an $n\times n$ matrix with characteristic polynomial $g(\lambda)$ and eigenvalues $\lambda_1, \lambda_2,\ldots, \lambda_y$.  Then
\[
\underset{1\leq i\leq y}\max{\left|\lambda_i\right|}\leq \underset{1\leq i\leq n}\max{\sum_{j=1}^n \left|m_{ij}\right|}.
\]

\end{lemma}

\begin{theorem}\label{main theorem}
Let $\mathcal{A},f_{\mathcal{A}}(n), M_{\mathcal{A}}$, and $\omega_k(m)$ be as above, with the additional assumption that there exists some odd $a_i\in\mathcal{A}$. Define
\[
 s_\mathcal{A}(r,m)=\sum_{n=m2^r}^{m2^{r+1}-1}f_{\mathcal{A}}(n).
\]
Let $\left|\mathcal{A}\right|$ denote the number of elements in the set $\mathcal{A}$.  Then for a fixed value of $m$,
\[
\lim_{r\to\infty}\frac{s_{\mathcal{A}}(r,m)}{|\mathcal{A}|^r}=c(\mathcal{A},m),
\]
for some constant $c(\mathcal{A},m)\in\mathbb{Q}$, so $s_{\mathcal{A}}(r,m)=c(\mathcal{A},m)\left|\mathcal{A}\right|^r(1+o(r))$.

\end{theorem}

\begin{proof}
Let $g(\lambda):=\operatorname{det}(M_{\mathcal{A}}-\lambda I)$ be the characteristic polynomial of $M_{\mathcal{A}}$ with eigenvalues $\lambda_1,\lambda_2,\ldots,\lambda_y$, where each $\lambda_i$ has multiplicity $e_i$. We can write
\[
 g(\lambda)=\sum_{k=0}^{a_z+1}\alpha_k\lambda^k.
\]
By Cayley-Hamilton, we know that $g\left(M_{\mathcal{A}}\right)=\mathbf{0}$.  Thus we have
\[
 \mathbf{0}=g\left(M_{\mathcal{A}}\right)=\sum_{k=0}^{a_z+1}\alpha_k M_{\mathcal{A}}^k
\]
and hence, for all $r$,
\[
 \mathbf{0}=\left(\sum_{k=0}^{a_z+1}\alpha_k M_{\mathcal{A}}^k\right)\omega_r(m)=\sum_{k=0}^{a_z+1}\alpha_k\omega_{r+k}(m).
\]
Since
\[
 \omega_{r+k}(m)=
\left(
\begin{array}{c}
f_{\mathcal{A}}(2^{r+k}m)\\
f_{\mathcal{A}}(2^{r+k}m-1)\\
\vdots\\
f_{\mathcal{A}}(2^{r+k}m-a_z)\\
\end{array}
\right),
\]
we have
\begin{equation}\label{answers h(r)}
\sum_{k=0}^{a_z+1}\alpha_kf(2^{r+k}m-j)=0
\end{equation}
for all $0\leq j\leq a_z$.

Let $I_r=\{2^r,2^r+1,2^r+2,\ldots,2^{r+1}-1\}$.  Then $I_r=2I_{r-1}\cup\left(2I_{r-1}+1\right)$.  Thus
\begin{align*}
 s_\mathcal{A}(r,m)&=\sum_{n=m2^r}^{m2^{r+1}-1}f_{\mathcal{A}}(n)\\
  &=\sum_{n=m2^{r-1}}^{m2^r-1}f_{\mathcal{A}}(2n)+f_{\mathcal{A}}(2n+1)\\
  &=\sum_{n=m2^{r-1}}^{m2^r-1}f_{\mathcal{A}}(n)+f_{\mathcal{A}}(n-b_2)+\cdots+f_{\mathcal{A}}(n-b_s)+f_{\mathcal{A}}(n-c_1)+\cdots+f_{\mathcal{A}}(n-c_t).
\end{align*}
Since
\[
 \sum_{n=m2^{r-1}}^{m2^r-1}f_\mathcal{A}(n-k)=\sum_{n=m2^{r-1}}^{m2^r-1}f_{\mathcal{A}}(n)+\sum_{j=1}^k\left(f_{\mathcal{A}}(m2^{r-1}-j)-f_\mathcal{A}(m2^r-j)\right),
\]
we deduce that
\begin{align*}
 s_{\mathcal{A}}(r,m)&=\left|\mathcal{A}\right|\sum_{n=m2^{r-1}}^{m2^r-1}f_{\mathcal{A}}(n)+h(r)\\
&=\left|\mathcal{A}\right|s_{\mathcal{A}}(r-1,m)+h(r),
\end{align*}
where
\[
h(r)=\sum_{i=2}^s\sum_{j=1}^{b_i} \left(f_{\mathcal{A}}(m2^{r-1}-j)-f_\mathcal{A}(m2^r-j)\right)\\
 +\sum_{i=1}^t\sum_{j=1}^{c_i} \left(f_{\mathcal{A}}(m2^{r-1}-j)-f_{\mathcal{A}}(m2^r-j)\right)
\]
and
\[
 \sum_{k=0}^{a_z+1}\alpha_kh(r+k)=0
\]
by Equation \eqref{answers h(r)}.

Thus we have an inhomogeneous recurrence relation for $s_{\mathcal{A}}(r,m)$ and will first consider the corresponding homogeneous recurrence relation 
\[
s_{\mathcal{A}}(r,m)=\left|\mathcal{A}\right|s_{\mathcal{A}}(r-1,m), 
\]
which has solution $s_{\mathcal{A}}(r,m)=c_1\left|\mathcal{A}\right|^r$.  Then the solution to our inhomogeneous recurrence relation is of the form
\[
 s_{\mathcal{A}}(r,m)=c_1\left|\mathcal{A}\right|^r+\sum_{i=1}^yp_i(\lambda_i,r),
\]
where
\[
p_i(\lambda_i, r)=\sum_{j=1}^{e_i}c_{ij}r^{j-1}\lambda_i^r.
\]

By Lemma \ref{eignevalues bounded by row sums}, $\left|\lambda_i\right|$ is bounded above by the maximum row sum of $M_{\mathcal{A}}$, which is at most $\left|\mathcal{A}\right|-1$ since all elements of $M_{\mathcal{A}}$ are either $0$ or $1$ and  by assumption not all elements have the same parity.  Hence the $c_1|\mathcal{A}|^r$ term dominates $s_{\mathcal{A}}(r,m)$ as $r\to\infty$, so
\[
\lim_{r\to\infty}\frac{s_{\mathcal{A}}(r,m)}{|\mathcal{A}|^r}=c_1.
\] 

Observe that
\[
\sum_{k=0}^{a_z+1}\alpha_k\sum_{i=1}^y p_i\left(\lambda_i,r+k\right)=0.
\]
Thus we can compute $\sum_{k=0}^{a_z+1}\alpha_ks_{\mathcal{A}}(r+k,m)$, and for sufficiently large $r$,
\[
 \sum_{k=0}^{a_z+1}\alpha_ks_{\mathcal{A}}(r+k,m)=c_1\sum_{k=0}^{a_z+1}\alpha_k\left|\mathcal{A}\right|^{r+k} +0=c_1\left|\mathcal{A}\right|^rg\left(\left|\mathcal{A}\right|\right).
\]
Then we can solve for $c_1$ to see that
\begin{equation}\label{c_1 formula}
 c_1=c(\mathcal{A},m):=\frac{\sum_{k=0}^{a_z+1}\alpha_ks_{\mathcal{A}}(r+k,m)}{\left|\mathcal{A}\right|^rg\left(\left|\mathcal{A}\right|\right)}.
\qedhere
\end{equation}

\end{proof}

\section{Examples}\label{Examples}

\begin{example}\label{018}
Let $\mathcal{A}=\{0,1,8\}$.  Then
\begin{equation}\label{018 even recurrence}
f_\mathcal{A}(2\ell)=f_\mathcal{A}(\ell)+f_\mathcal{A}(\ell-4)
\end{equation}
and
\begin{equation}\label{018 odd recurrence}
f_\mathcal{A}(2\ell+1)=f_\mathcal{A}(\ell),
\end{equation}
so
\begin{equation*}
\left(
\begin{array}{c}
f_{\mathcal{A}}(2^{k+1}m)\\
f_{\mathcal{A}}(2^{k+1}m-1)\\
f_{\mathcal{A}}(2^{k+1}m-2)\\
f_{\mathcal{A}}(2^{k+1}m-3)\\
f_{\mathcal{A}}(2^{k+1}m-4)\\
f_{\mathcal{A}}(2^{k+1}m-5)\\
f_{\mathcal{A}}(2^{k+1}m-6)\\
f_{\mathcal{A}}(2^{k+1}m-7)\\
f_{\mathcal{A}}(2^{k+1}m-8)\\
\end{array}
\right)
=
\left(
\begin{array}{c c c c c c c c c}
1 &0 &0 &0 &1 &0 &0 &0 &0\\
0 &1 &0 &0 &0 &0 &0 &0 &0\\
0 &1 &0 &0 &0 &1 &0 &0 &0\\
0 &0 &1 &0 &0 &0 &0 &0 &0\\
0 &0 &1 &0 &0 &0 &1 &0 &0\\
0 &0 &0 &1 &0 &0 &0 &0 &0\\
0 &0 &0 &1 &0 &0 &0 &1 &0\\
0 &0 &0 &0 &1 &0 &0 &0 &0\\
0 &0 &0 &0 &1 &0 &0 &0 &1\\
\end{array}
\right)
\left(
\begin{array}{c}
f_{\mathcal{A}}(2^km)\\
f_{\mathcal{A}}(2^km-1)\\
f_{\mathcal{A}}(2^km-2)\\
f_{\mathcal{A}}(2^km-3)\\
f_{\mathcal{A}}(2^km-4)\\
f_{\mathcal{A}}(2^km-5)\\
f_{\mathcal{A}}(2^km-6)\\
f_{\mathcal{A}}(2^km-7)\\
f_{\mathcal{A}}(2^km-8)\\
\end{array}
\right).
\end{equation*}
If $M_\mathcal{A}$ is the matrix above, then $\omega_{k+1}(m)=M_\mathcal{A}\omega_k(m)$.  The characteristic polynomial of $M_\mathcal{A}$ is 
\begin{equation}\label{char poly of 018}
g(x)=1 - 3 x + 3 x^2 - 3 x^3 + 6 x^4 - 6 x^5 + 3 x^6 - 3 x^7 + 3 x^8 - x^9.
\end{equation}
We then compute
\begin{align*}
s_\mathcal{A}(3,1)&-3s_\mathcal{A}(4,1)+3s_\mathcal{A}(5,1)-3s_\mathcal{A}(6,1)+6s_\mathcal{A}(7,1)-6s_\mathcal{A}(8,1)\\
&+3s_\mathcal{A}(9,1)-3s_\mathcal{A}(10,1)+3s_\mathcal{A}(11,1)-s_\mathcal{A}(12,1)\\
&=-59184
\end{align*}
Using the formula from Theorem \ref{main theorem}, we see that
\[
c(\mathcal{A},1)=\frac{-59184}{g(3)\cdot27}=\frac{-59184}{-5408\cdot27}=\frac{137}{338}.
\]
\end{example}

\begin{example}\label{013}
Let $\mathcal{A}=\{0,1,3\}$.  Then
\begin{equation}\label{0,1,3 even recurrence}
f_\mathcal{A}(2\ell)=f_\mathcal{A}(\ell)
\end{equation}
and
\begin{equation}\label{0,1,3 odd recurrence}
f_\mathcal{A}(2\ell +1)=f_\mathcal{A}(\ell)+f_\mathcal{A}(\ell-1),
\end{equation}
so
\[
\left(
\begin{array}{c}
f_{\mathcal{A}}(2^{k+1}m)\\
f_{\mathcal{A}}(2^{k+1}m-1)\\
f_{\mathcal{A}}(2^{k+1}m-2)\\
f_{\mathcal{A}}(2^{k+1}m-3)\\
\end{array}
\right)
=
\left(
\begin{array}{c c c c}
1 &0 &0 &0\\
0 &1 &1 &0\\
0 &1 &0 &0\\
0 &0 &1 &1\\
\end{array}
\right)
\left(
\begin{array}{c}
f_{\mathcal{A}}(2^km)\\
f_{\mathcal{A}}(2^km-1)\\
f_{\mathcal{A}}(2^km-2)\\
f_{\mathcal{A}}(2^km-3)\\
\end{array}
\right).
\]
Hence $M_{\mathcal{A}}=\left(
\begin{array}{c c c c}
1 &0 &0 &0\\
0 &1 &1 &0\\
0 &1 &0 &0\\
0 &0 &1 &1\\
\end{array}
\right)$
satisfies $\omega_{k+1}(m)=M_{\mathcal{A}}\omega_k(m)$.  The characteristic polynomial of $M_{\mathcal{A}}$ is 
\begin{equation}\label{char poly of 0,1,3}
g(x)=(x-1)^2(x^2-x-1).
\end{equation}

Let $F_k$ denote the $k$-th Fibonacci number.  Then 
\begin{equation}\label{Fibonacci 013}
f_\mathcal{A}(2^k-1)=F_{k+1}
\end{equation} 
for all $k\geq 0$.  This can be shown by using induction and Equations \eqref{0,1,3 even recurrence}  and \eqref{0,1,3 odd recurrence}.

Considering the summatory function with $m=1$ and using Equations \eqref{0,1,3 even recurrence},\eqref{0,1,3 odd recurrence}, and \eqref{Fibonacci 013}, we see that
\begin{align*}
s_{\mathcal{A}}(r,1)&=\sum_{n=2^r}^{2^{r+1}-1} f_\mathcal{A}(n)\\
&=\sum_{n=2^{r-1}}^{2^r-1}\left(f_\mathcal{A}(2n)+f_\mathcal{A}(2n+1)\right)\\
&=\sum_{n=2^{r-1}}^{2^r-1}\left(f_\mathcal{A}(n)+f_\mathcal{A}(n)+f_\mathcal{A}(n-1)\right)\\
&=2s_\mathcal{A}(r-1,1)+\sum_{n=2^{r-1}}^{2^r-1}f_\mathcal{A}(n-1)\\
&=2s_\mathcal{A}(r-1,1)+\sum_{n=2^{r-1}}^{2^r-1}f_\mathcal{A}(n)+f_\mathcal{A}(2^{r-1}-1)-f_\mathcal{A}(2^r-1)\\
&=3s_\mathcal{A}(r-1,1)+f_\mathcal{A}(2^{r-1}-1)-f_\mathcal{A}(2^r-1)\\
&=3s_\mathcal{A}(r-1,1)+F_r-F_{r+1}\\
&=3s_\mathcal{A}(r-1,1)-F_{r-1}.
\end{align*}

This is an inhomogeneous recurrence relation for $s_\mathcal{A}(r,1)$.  We first consider the corresponding homogeneous recurrence relation $s_\mathcal{A}(r,1)=3s_\mathcal{A}(r-1,1)$, which has solution
\[
s_\mathcal{A}(r,1)=c_13^r,
\]
for some $c_1$ in $\mathbb{Q}$.  Recall that the characteristic polynomial $g(x)$ of $M_{\mathcal{A}}$ has roots $1,\phi$, and $\bar{\phi}$, where the first has multiplicity $2$ and the others have multiplicity $1$.  Hence the solution to the inhomogeneous recurrence relation is
\begin{equation}\label{solution to inhomogeneous example 013}
s_\mathcal{A}(r,1)=c_13^r+c_2\phi^r+c_3\bar{\phi}^r+c_4(1)^r+c_5r(1)^r,
\end{equation}
where $c_2,c_3, c_4, c_5\in\mathbb{Q}$.  Observe that the $c_13^r$ summand will dominate as $r\to\infty$, so
\[
 \lim_{r\to\infty}\frac{s_{\mathcal{A}}(r,1)}{3^r}=c_1
\]
and $s_{\mathcal{A}}(r,1)\approx c_13^r$.

Using Equations \eqref{char poly of 0,1,3} and \eqref{solution to inhomogeneous example 013}, we can compute $c_1$ as 
\begin{align*}
s_\mathcal{A}(r+2,1)-s_\mathcal{A}(r+1,1)-s_\mathcal{A}(r,1)&=c_13^r(3^2-3-1)+c_2\phi^r(\phi^2-\phi-1)\\
&\quad+c_3\bar{\phi}^r(\bar{\phi}^2-\bar{\phi}-1)+c_4(1^2-1-1)\\
&\quad+c_5(r+2-(r+1)-r)\\
&=c_13^r\cdot 5-c_4-c_5(r-1).
\end{align*}
Plugging in $r=2$, $r=1$, and $r=0$ and computing sums, we see that $c_1=4/5$.  Hence 
\[
 \lim_{r\to\infty}\frac{s_{\mathcal{A}}(r,1)}{3^r}=\frac{4}{5}
\]
and $s_{\mathcal{A}}(r,1)\approx \frac{4}{5}(3)^r$.
\end{example}

\begin{example}\label{023}
Let $\mathcal{\tilde{A}}=\{0,2,3\}$.  Then
\begin{equation}\label{0,2,3 even recurrence}
f_\mathcal{\tilde{A}}(2\ell)=f_\mathcal{\tilde{A}}(\ell)+f_\mathcal{\tilde{A}}(\ell-1)
\end{equation}
and
\begin{equation}\label{0,2,3 odd recurrence}
f_\mathcal{\tilde{A}}(2\ell +1)=f_\mathcal{\tilde{A}}(\ell-1),
\end{equation}
so
\[
\left(
\begin{array}{c}
f_{\mathcal{\tilde{A}}}(2^{k+1}m)\\
f_{\mathcal{\tilde{A}}}(2^{k+1}m-1)\\
f_{\mathcal{\tilde{A}}}(2^{k+1}m-2)\\
f_{\mathcal{\tilde{A}}}(2^{k+1}m-3)\\
\end{array}
\right)
=
\left(
\begin{array}{c c c c}
1 &1 &0 &0\\
0 &0 &1 &0\\
0 &1 &1 &0\\
0 &0 &0 &1\\
\end{array}
\right)
\left(
\begin{array}{c}
f_{\mathcal{\tilde{A}}}(2^km)\\
f_{\mathcal{\tilde{A}}}(2^km-1)\\
f_{\mathcal{\tilde{A}}}(2^km-2)\\
f_{\mathcal{\tilde{A}}}(2^km-3)\\
\end{array}
\right).
\]
Hence $M_{\mathcal{\tilde{A}}}=\left(
\begin{array}{c c c c}
1 &1 &0 &0\\
0 &0 &1 &0\\
0 &1 &1 &0\\
0 &0 &0 &1\\
\end{array}
\right)$
satisfies $\omega_{k+1}(m)=M_{\mathcal{\tilde{A}}}\omega_k(m)$.  The characteristic polynomial of $M_{\mathcal{\tilde{A}}}$ is 
\begin{equation}\label{char poly of 0,2,3}
g(x)=(x-1)^2(x^2-x-1).
\end{equation}

Let $F_k$ denote the $k$-th Fibonacci number.  Then 
\begin{equation}\label{Fibonacci}
f_\mathcal{\tilde{A}}(2^k-1)=F_{k-1}
\end{equation} 
for all $k\geq 1$.  This can be shown by using induction and Equations \eqref{0,2,3 even recurrence}  and \eqref{0,2,3 odd recurrence} to prove that $f_\mathcal{\tilde{A}}(2^k-2)=F_k$ for all $k\geq 2$ and observing that Equation \eqref{0,2,3 odd recurrence} gives $f_\mathcal{\tilde{A}}(2^k-1)=f_\mathcal{\tilde{A}}(2^{k-1}-2)$. 

Considering the summatory function with $m=1$ and using Equations \eqref{0,2,3 even recurrence},\eqref{0,2,3 odd recurrence}, and \eqref{Fibonacci} and manipulations similar to those in Example \ref{013}, we see that
\[
s_{\mathcal{\tilde{A}}}(r,1)=3s_\mathcal{\tilde{A}}(r-1,1)-2F_{r-3}.
\]

Again, the corresponding homogeneous recurrence relation has solution
\[
s_\mathcal{\tilde{A}}(r,1)=c_13^r,
\]
for some $c_1$ in $\mathbb{Q}$, and we can use Equation \eqref{char poly of 0,2,3} to see that the solution to the inhomogeneous recurrence relation is
\begin{equation}\label{solution to inhomogeneous example}
s_\mathcal{\tilde{A}}(r,1)=c_13^r+c_2\phi^r+c_3\bar{\phi}^r+c_4(1)^r+c_5r(1)^r,
\end{equation}
where $c_2,c_3, c_4, c_5\in\mathbb{Q}$.  Observe that the $c_13^r$ summand will dominate as $r\to\infty$, so
\[
 \lim_{r\to\infty}\frac{s_{\mathcal{\tilde{A}}}(r,1)}{3^r}=c_1
\]
and $s_{\mathcal{\tilde{A}}}(r,1)\approx c_13^r$.

Using Equations \eqref{char poly of 0,2,3} and \eqref{solution to inhomogeneous example}, we can compute $c_1$ as 
\[
s_\mathcal{\tilde{A}}(r+2,1)-s_\mathcal{\tilde{A}}(r+1,1)-s_\mathcal{\tilde{A}}(r,1)=c_13^r\cdot 5-c_4-c_5(r-1).
\]
Plugging in $r=2$, $r=1$, and $r=0$ and computing sums, we see that $c_1=2/5$.  Hence 
\[
 \lim_{r\to\infty}\frac{s_{\mathcal{\tilde{A}}}(r,1)}{3^r}=\frac{2}{5}
\]
and $s_{\mathcal{\tilde{A}}}(r,1)\approx \frac{2}{5}(3)^r$.
\end{example}

In Example \ref{013}, we had $\mathcal{A}=\{0,1,3\}$, and in Example \ref{023}, we had $\mathcal{\tilde{A}}=\{0,2,3\}=\{3-3,3-1,3-0\}$.  We found $c(\mathcal{A},1)$ in Example \ref{013} and $c(\mathcal{\tilde{A}},1)$ in Example \ref{023} and can observe that they have the same denominator.

Given a set $\mathcal{A}=\{0,a_1,\ldots,a_z\}$, let $\mathcal{\tilde{A}}$ be
\[
\mathcal{\tilde{A}}:=\{0,a_z-a_{z-1},\ldots,a_z-a_1,a_z\}.
\] 
The following chart displays the value $c(\mathcal{A},1)$ for various sets $\mathcal{A}$ and their corresponding sets $\mathcal{\tilde{A}}$, where $s_{\mathcal{A}}(r,1)\approx c(\mathcal{A},1)|\mathcal{A}|^r$.  Note that in all cases the denominator of $c(\mathcal{A},1)$ is the same as that of $c(\mathcal{\tilde{A}},1)$.  The following theorem will show that this holds for all $\mathcal{A}$.

\begin{table}[H]
\begin{center}
\renewcommand{\arraystretch}{2}
\begin{tabular}{l l l l| l l l}
$\mathcal{A}$ & $c(\mathcal{A},1)$ &$\operatorname{N}(c(\mathcal{A},1))$ & &$\mathcal{\tilde{A}}$ &$c(\mathcal{\tilde{A}},1)$ &$\operatorname{N}(c(\mathcal{\tilde{A}},1))$\\ \hline
$\{0,1,2,4\}$ & $\frac{7}{11}$ &$0.636$ & & $\{0,2,3,4\}$ & $\frac{3}{11}$ &$0.273$\\
$\{0,1,3,4\}$ & $\frac{1}{2}$ &$0.500$ & & $\{0,1,3,4\}$ & $\frac{1}{2}$ &$0.500$\\
$\{0,2,3,6\}$ & $\frac{33}{149}$ &$0.221$ & & $\{0,3,4,6\}$ & $\frac{21}{149}$ &$0.141$\\
$\{0,1,6,9\}$ & $\frac{6345}{28670}$ &$0.221$ & & $\{0,3,8,9\}$ & $\frac{2007}{28670}$ &$0.070$\\
$\{0,1,7,9\}$ & $\frac{2069}{10235}$ &$0.202$ & & $\{0,2,8,9\}$ & $\frac{1023}{10235}$ &$0.100$\\
$\{0,4,5,6,9\}$ & $\frac{4044}{83753}$ &$0.048$ & & $\{0,3,4,5,9\}$ & $\frac{6716}{83753}$ & $0.080$\\
\end{tabular}
\end{center}\caption{$c(\mathcal{A},1)$ for various sets $\mathcal{A}$ and $\mathcal{\tilde{A}}$}\label{sets of size 4 and 5 growth coefficients}
\end{table}

\begin{theorem}
Let $\mathcal{A}, f_{\mathcal{A}}(n), M_{\mathcal{A}}=[m_{\alpha,\beta}]$, and $\mathcal{\tilde{A}}$ be as above, with $0\leq \alpha, \beta \leq a_z$.  Let $M_{\mathcal{\tilde{A}}}=\left[m^{'}_{\alpha,\beta}\right]$ be the $(a_z+1) \times (a_z+1)$ matrix such that
\[
\left(
\begin{array}{c}
f_{\mathcal{\tilde{A}}}(2n)\\
f_{\mathcal{\tilde{A}}}(2n-1)\\
\vdots\\
f_{\mathcal{\tilde{A}}}(2n-a_z)
\end{array}
\right)
=M_{\mathcal{\tilde{A}}}
\left(
\begin{array}{c}
f_{\mathcal{\tilde{A}}}(n)\\
f_{\mathcal{\tilde{A}}}(n-1)\\
\vdots\\
f_{\mathcal{\tilde{A}}}(n-a_z)
\end{array}
\right).
\]
Then $m_{\alpha,\beta}=m^{'}_{a_z-\alpha,a_z-\beta}$.
\end{theorem}

\begin{proof}
Recall that we can write
\[
\mathcal{A}:=\{0,2b_2,\ldots,2b_s,2c_1+1,\ldots,2c_t+1\}, 
\]
so that
\[
f_{\mathcal{A}}(2n-2j)=f_{\mathcal{A}}(n-j)+f_{\mathcal{A}}(n-j-b_2)+\cdots+f_{\mathcal{A}}(n-j-b_s)
\]
and
\[
f_{\mathcal{A}}(2n-2j-1)=f_{\mathcal{A}}(n-j-c_1-1)+\cdots+f_{\mathcal{A}}(n-j-c_t-1)
\]
for $j$ sufficiently large.  

Then $m_{\alpha,\beta}=1$ if and only if $f_{\mathcal{A}}(n-\beta)$ is a summand in the recursive sum that expresses $f_{\mathcal{A}}(2n-\alpha)$, which happens if and only if $2n-\alpha=2(n-\beta)+K$, where $K\in\mathcal{A}$, and this is equivalent to $2\beta-\alpha$ belonging to $\mathcal{A}$.

Now $m^{'}_{a_z-\alpha,a_z-\beta}=1$ if and only if $f_{\mathcal{\tilde{A}}}(n-(a_z-\beta))$ is a summand in the recursive sum that expresses $f_{\mathcal{\tilde{A}}}(2n-(a_z-\alpha))$, which happens if and only if $2n-(a_z-\alpha)=2(n-(a_z-\beta))+\tilde{K}$, where $\tilde{K}\in\mathcal{\tilde{A}}$.  This means that $a_z+\alpha-2\beta=\tilde{K}$, which gives $2\beta-\alpha\in\mathcal{A}$.
\end{proof}

Thus $M_\mathcal{A}=S^{-1}M_{\mathcal{\tilde{A}}}S$, where
\[
S=
\left(
\begin{array}{c c c c c}
0 &0 &\cdots &0 &1\\
0 &0 &\cdots &1 &0\\
\vdots &\vdots & &\vdots &\vdots\\
0 &1 &\cdots &0 &0\\
1 &0 &\cdots &0 &0
\end{array}
\right),
\]
so $M_{\mathcal{A}}$ and $M_\mathcal{\tilde{A}}$ have the same characteristic polynomial, {\cite[1.3.3]{Matrix Theory}}.  We see that $c(\mathcal{A},m)$ and $c(\mathcal{\tilde{A}},m)$ have the same denominator.

\section{Open Questions}

A nicer formula for $c(\mathcal{A},m)$ than that given in Equation \eqref{c_1 formula} is desired and seems likely.  To that end, we have computed values of $c(\mathcal{A})$ for a variety of sets $\mathcal{A}$ but have not been able to detect any patterns. Table \ref{sets of size 3 growth coefficients} shows $c(\mathcal{A},1)$ for all sets of the form $\mathcal{A}=\{0,1,t\}$, where $2\leq t\leq 17$, and we have obtained the following bounds on $c(\mathcal{A},1)$ for sets $\mathcal{A}$ of this form.

Let $t\in\mathbb{N}$ with $t>1$ and $\mathcal{A}=\{0,1,t\}$.  Choose $k$ such that $2^k<t\leq 2^{k+1}$.  Recall that $f_\mathcal{A}(s)$ is the number of ways to write $s$ in the form
\[
s=\sum_{i=0}^{\infty}\epsilon_i2^i, \text{ where }\epsilon_i\in\{0,1,t\}.
\]
Then 
\[
s_\mathcal{A}(r,1)=\sum_{n=2^r}^{2^{r+1}-1}f_\mathcal{A}(n)\approx c(\mathcal{A},1)3^r,
\]
as shown in Theroem \ref{main theorem}.  Thus
\begin{align*}
\sum_{s=1}^{2^n-1}f_\mathcal{A}(s)&=\sum_{j=0}^{n-1}\sum_{s=2^j}^{2^{j+1}-1}f_\mathcal{A}(s)\approx\sum_{j=0}^{n-1}c(\mathcal{A},1)3^j\\
&=c(\mathcal{A},1)\left(\frac{3^n-1}{2}\right)\approx \frac{1}{2}c(\mathcal{A},1)3^n.
\end{align*}

Consider choosing $\epsilon_i\in\{0,1,t\}$ for $0\leq i\leq n-k-3$ and $\epsilon_i\in\{0,1\}$ for $n-k-2\leq i\leq n-2$.  Then
\begin{align*}
\sum_{i=0}^{n-2} \epsilon_i 2^i&\leq t+t\cdot2+t\cdot2^2+\cdots+t\cdot2^{n-k-3}+2^{n-k-2}+2^{n-k-1}+\cdots+2^{n-2}\\
&<t\cdot2^{n-k-2}+2^{n-1}-1\\
&\leq2^{k+1}\cdot2^{n-k-2}+2^{n-1}-1\\
&=2^n-1\\
&<2^n.
\end{align*}
There are $3^{n-k-2}\cdot 2^{k+1}$ such sums, and each of them is counted in $\sum_{s=1}^{2^n-1}f_\mathcal{A}(s)$.  Thus 
\[
\frac{1}{2}c(\mathcal{A},1)3^n\geq 3^{n-k-2}\cdot 2^{k+1}=3^n\cdot \frac{2^{k+1}}{3^{k+2}},
\]
and so $c(\mathcal{A},1)\geq \left(\frac{2}{3}\right)^{k+2}$.

Now suppose there exists some $i_0\geq n-k$ such that $\epsilon_{i_0}=t$.  Then
\[
\sum_{i=0}^{\infty}\epsilon_i2^i\geq t\cdot2^{i_0}\geq t\cdot2^{n-k}>2^k2^{n-k}=2^n.
\]
Thus the sums counted in $\sum_{s=1}^{2^n-1}f_\mathcal{A}(s)$ all have the propety that $\epsilon_i\in\{0,1\}$ for $n-k\leq i\leq n-1$, and there are $3^{n-k}\cdot 2^k$ such sums.
Hence $3^{n-k}\cdot 2^k\geq \frac{1}{2}c(\mathcal{A},1)\cdot 3^n$ and $\frac{2^{k+1}}{3^k}\geq c(\mathcal{A},1)$.

Combining the above, we see that
\[
\frac{2^{k+1}}{3^k}\cdot \frac{2}{9}\leq c(\mathcal{A},1)\leq \frac{2^{k+1}}{3^k}.
\]
To compare these bounds with Table \ref{sets of size 3 growth coefficients}, note that if $8<t\leq15$, then $k=3$, and we have 
\[
0.132\leq c(\mathcal{A},1)\leq 0.593
\]
for $\mathcal{A}=\{0,1,t\}$, with $t$ in this range.

We have also computed $c(\mathcal{A},1)$ for some sets with $|\mathcal{A}|=4$ and $|\mathcal{A}|=5$, and that data is contained in Table \ref{sets of size 4 and 5 growth coefficients}.  Larger sets have not been considered because computations become increasingly tedious as the cardinality of $\mathcal{A}$ grows.

\begin{table}
\begin{center}
\renewcommand{\arraystretch}{1.5}
\begin{tabular}{l l l l| l l l}
$\mathcal{A}$ & $c(\mathcal{A},1)$ &$\operatorname{N}(c(\mathcal{A},1))$ & &$\mathcal{A}$ &$c(\mathcal{A},1)$ &$\operatorname{N}(c(\mathcal{{A}},1))$\\ \hline
$\{0,1,2\}$ & $1$ &$1.000$ & & $\{0,1,3\}$ & $\frac{4}{5}$ &$0.800$\\
$\{0,1,4\}$ & $\frac{5}{8}$ &$0.625$ & & $\{0,1,5\}$ & $\frac{14}{25}$ &$0.560$\\
$\{0,1,6\}$ & $\frac{35}{71}$ &$0.493$ & & $\{0,1,7\}$ & $\frac{176}{391}$ &$0.450$\\
$\{0,1,8\}$ & $\frac{137}{338}$ &$0.405$ & & $\{0,1,9\}$ & $\frac{1448}{3775}$ &$0.384$\\
$\{0,1,10\}$ & $\frac{1990}{5527}$ &$0.360$ & & $\{0,1,11\}$ & $\frac{3223}{9476}$ &$0.340$\\
$\{0,1,12\}$ & $\frac{2020}{6283}$ &$0.322$ & &  $\{0,1,13\}$ & $\frac{47228}{154123}$ &$0.306$\\
$\{0,1,14\}$ & $\frac{35624}{122411}$ &$0.291$ & & $\{0,1,15\}$ & $\frac{699224}{2501653}$ &$0.280$\\
$\{0,1,16\}$ & $\frac{68281}{256000}$ &$0.267$ & & $\{0,1,17\}$ & $\frac{38132531}{146988000}$ &$0.259$\\
\end{tabular}
\end{center}\caption{$c(\mathcal{A},1)$ for all sets of the form $\mathcal{A}=\{0,1,t\}$, where $2\leq t\leq 17$}.\label{sets of size 3 growth coefficients} 
\end{table}

\section{Acknowledgements}
The author acknowledges support from National Science Foundation grant DMS 08-38434 ``EMSW21-MCTP: Research Experience for Graduate Students.''  The results in this paper were part of the author's doctoral dissertation \cite{my thesis} at the University of Illinois at Urbana-Champaign.  The author wishes to thank Professor Bruce Reznick for his time, ideas, and encouragement.

\bigskip
\hrule
\bigskip

\noindent 2010 {\it Mathematics Subject Classification}:
Primary 11A63.

\noindent \emph{Keywords: } 
digital representation, non-standard binary representation, summatory function.

\bigskip
\hrule
\bigskip

\vspace*{+.1in}
\noindent
Received August 25, 2015

\end{document}